\documentclass[reqno,oneside,12pt]{amsart}

\usepackage[T1]{fontenc}
\usepackage{times,mathptm}
\usepackage{amssymb,epsfig,verbatim,xypic}

\theoremstyle{plain}

\newtheorem{thm}{Theorem}[section]

\newtheorem{pro}[thm]{Proposition}
\newtheorem{lem}[thm]{Lemma}
\newtheorem{proposition-principale}[thm]{Proposition principale}
\newtheorem{thm-principal}{Th\'eor\`eme principal}[section]

\theoremstyle{definition}

\newtheorem{rem}[thm]{Remark}

\newenvironment{thm-A}
{\noindent{\bf Theorem A (Urech).--}\it}{\\}

\newenvironment{thm-M}
{\noindent{\bf Main Theorem.}\it}{\\}

\newenvironment{thm-AA}
{\noindent{\bf Theorem A'.}\it}{\\}

\newenvironment{thm-B}
{\noindent{\bf Theorem B.--}\it}{\\}

\newenvironment{thm-C}
{\noindent{\bf Theorem C.--}\it}{\\}

\newenvironment{thm-BP}
{\noindent{\bf Bell-Poonen Theorem.--}\it}{\\}

\newenvironment{thm-BB}
{\noindent{\bf Theorem B'.}\it}

\def\C{\mathbf{C}}
\def\bfk{\mathbf{k}}

\def\bfF{{\mathbf{F}}}

\def\R{\mathbf{R}}
\def\Q{\mathbf{Q}}

\def\Z{\mathbf{Z}}
\def\N{\mathbf{N}}

\def\Id{{\mathsf{Id}}}
\def\bfx{{\mathbf{x}}}
\def\bfa{{\mathbf{a}}}
\def\bfb{{\mathbf{b}}}
\def\bfc{{\mathbf{c}}}
\def\bft{{\mathbf{t}}}

\def\Gr{{\mathcal{G}}}
\def\Lin{{\text{Lin}}}
\def\Diff{{\sf{Diff}}}

\def\ux{{\mathbf{x}}}
\def\uy{{\mathbf{y}}}

\def\bbP{\mathbb{P}}
\def\bbA{\mathbb{A}}

\def\Aut{{\sf{Aut}}}
\def\Bir{{\sf{Bir}}}

\def\GL{{\sf{GL}}\,}

\def\U{{\sf{U}}\,}
\def\V{{\sf{V}}\,}

\def\End{{\sf{End}}\,}

\def\Exc{{\text{Exc}}}

\addtocounter{section}{0}             
\numberwithin{equation}{section}       

\begin{document}

\setlength{\baselineskip}{0.54cm}        
%
%
\title[Degree sequences]
{On degrees of birational mappings}
\date{2018}
\author{Serge Cantat and Junyi Xie}
\address{Univ Rennes, CNRS, IRMAR - UMR 6625, F-35000 Rennes, France}
\email{serge.cantat@univ-rennes1.fr, junyi.xie@univ-rennes1.fr}
 
%
%

%
%

%
%

\begin{abstract} 

We prove that the degrees of the iterates $\deg(f^n)$ of a birational 
map satisfy $\liminf(\deg(f^n))<+\infty$ if and only if the sequence $\deg(f^n)$
is bounded, and that the growth of $\deg(f^n)$ cannot be arbitrarily slow, unless $\deg(f^n)$ is bounded. 
\end{abstract}

\maketitle

\section{Degree sequences}\label{par:Intro}

Let $\bfk$ be a field. 
Consider a projective variety $X$, a polarization $H$ of $X$ (given by hyperplane sections 
of $X$ in some embedding $X\subset \bbP^N$), and a birational transformation $f$ of $X$, 
all defined over the field $\bfk$. Let $k$ be
the dimension of~$X$. The {\bf{degree}} of $f$ with respect to the polarization $H$ is the integer
\begin{equation}
\deg_H(f)=(f^*H)\cdot H^{k-1}
\end{equation}
where $f^*H$ is the total transform of $H$, and $(f^*H)\cdot H^{k-1}$ is the intersection product of $f^*H$
with $k-1$ copies of $H$. The degree is a positive integer, which we shall simply denote by $\deg(f)$, even
if it depends on $H$. When $f$ is a birational transformation of the projective space $\bbP^k$ and the polarization 
is given by ${\mathcal{O}}_{\bbP^k}(1)$ (i.e. by hyperplanes $H\subset \bbP^k$), then $\deg(f)$ is  the degree of the
homogeneous polynomial formulas defining $f$ in homogeneous coordinates. 

The degrees are submultiplicative, in the following sense:  
\begin{equation}\label{eq:sub-multi}
\deg(f\circ g)\leq c_{X,H}  \deg(f) \deg(g)
\end{equation}
for some positive constant $c_{X,H}$ and for every pair of birational transformations. 
Also, if the polarization $H$ is changed into another polarization $H'$, there is a positive constant $c$ which depends on $X$, $H$ and $H'$ but not on $f$, such 
that 
\begin{equation}\label{eq:change-of-polarization}
\deg_{H}(f)\leq c\deg_{H'}(f)
\end{equation}  
We refer to \cite{Dinh-Sibony:2005, NguyenBD:2017, TTTruong} for these fundamental properties.

The {\bf{degree sequence}} of $f$ is the sequence $(\deg(f^n))_{n\geq 0}$; it plays an important role in 
the study of the dynamics and the geometry of $f$. There are infinitely, but only countably many degree sequences 
(see~\cite{Bonifant2000,Urech}); unfortunately, not much is known on these sequences when $\dim(X)\geq 3$ (see~\cite{Blanc-Cantat, Diller-Favre} for $\dim(X)=2$).  In this article, we obtain the following basic results.
\begin{itemize}
\item The sequence $(\deg(f^n))_{n\geq 0}$ is bounded if and only if it is bounded along an infinite subsequence (see Theorems A and B in \S~\ref{par:Auto} and \S~\ref{par:BirationalTransformations}).
\item If the sequence $(\deg(f^n))_{n\geq 0}$ is unbounded, then its growth can not be arbitrarily slow; for instance, $\max_{0\leq j\leq n}\deg(f^j)$
is asymptotically bounded from below by the inverse of the diagonal Ackermann function when $X=\bbP^k_\bfk$ (see Theorem C in \S~\ref{par:Ackermann} for a better result).
 \end{itemize}
We focus on birational transformations because 
 a rational dominant transformation which is not birational has a topological degree $\delta > 1$, and this forces an exponential 
growth of the degrees: 
$1 < \delta^{1/k} \leq \lim_n( \deg(f^n)^{1/n})
$
where $k=\dim(X)$ (see~\cite{Dinh-Sibony:2005} and \cite{Pano}, pages 120--126).

%
%

\section{Automorphisms of the affine space}\label{par:Auto}

%
%

We start with the simpler case of automorphisms of the affine space; the goal of this section is to introduce
a $p$-adic method to study degree sequences.\\

\begin{thm-A}  
Let $f$ be an automorphism of the affine space $\bbA^k_\bfk$. If 
$\deg(f^n)$ is bounded along an infinite subsequence, then it is 
bounded.
\end{thm-A}

\vspace{-0.5cm}



\subsection{Urech's proof} In~\cite{Urech}, Urech proves a stronger result. Writing his proof in an intrinsic way, we extend it to  affine varieties:

\begin{thm}\label{thm:affbound}
Let $X=Spec\, A$ be an irreducible affine variety of dimension $k$ over the field $\bfk$. 
Let $f:X \to X$ be an automorphism. If  $(\deg(f^n))$ is unbounded there exists $\alpha>0$ such that 
$\#\{n\geq 0|\,\, \deg(f^n)\leq d\}\leq \alpha d^k$; in particular,  $\max_{0\leq j\leq n} \deg(f^j)$ is bounded from below by 
$(n/\alpha)^{1/k}$.
\end{thm}

Here, the degree of $f^n$, depends on the choice of a projective compactification $Y$ of $X$ and an ample line bundle $L$ on $Y$. 
However, by Equation~\eqref{eq:change-of-polarization}, the statement of Theorem \ref{thm:affbound} does not depend on the choice of $(Y,L)$.
Since automorphisms of $X$ always lift to its normalization, we may assume that $X$ is normal.
To prove this theorem, we shall introduce another equivalent notion of degree. 

\subsubsection{Degrees on affine varieties}\label{par:degrees-affine} Consider $X$ as a subvariety $X\subseteq \bbA^N\subseteq \bbP^N.$  Let $\bar{X}$ be the Zariski closure of $X$ in $\bbP^N$ and $H_1:=\bbP^N\setminus \bbA^N$ be the hyperplane at infinity.
Let $\pi:Y\to \bar{X}$ be its normalization: 
$Y$ is a normal projective compactification of $X$. Since $\pi:Y\to \bar{X}$ is finite, there exists $m\geq 1$ such (i) $H:=\pi^*(mH_1|_{\bar{X}})$ is very ample on $Y$ and (ii) $H$ is projectively normal on $Y$ i.e.  for every $n\geq 0$, the morphism 
$(H^0(Y,H))^{\otimes n}\to H^0(Y,nH)$ is surjective.

If $P\in A$ is a regular function on $X$, we extend it as a rational function on $Y$, we denote by $(P)=(P)_0-(P)_\infty$ the divisor defined by $P$ on $Y$, and  we define 
\begin{eqnarray}
\Delta(P) & = & \min\{d\geq 0|\,\,(P)+dH\geq 0 \text{ on } Y\}, \\
A_d & = & \{P\in A|\,\, \Delta(P)\leq d\}, \quad (\forall d\geq 0).
\end{eqnarray}
Then $A=\cup_{d\geq 0}A_d.$ Since $Y\setminus X$ is the support of $H$,
we get an isomorphism 
$i_n:H^0(Y,nH)\to A_n\subseteq A$ for every $n\geq 0$.
Thus, $A_1$ generates $A$ and the morphism $A_1^{\otimes n}\to A_n$ is surjective.
Now we define 
\begin{equation}
\deg^H(f)=\min\{m\geq 0|\,\, \Delta({f^*P})\leq m \text{ for every } P\in A_1\}.
\end{equation}
For every $P\in A_n$, we can write $P=\sum_{i=1}^l g_{1,i}\dots g_{1,n}$ for some $g_{i,j}\in A_1$.
We get $f^*P=\sum_{i=1}^l f^*g_{1,i}\dots f^*g_{1,n}\in A_{\deg^H(f)n}$ and
\begin{equation}
\Delta(f^*P)\leq \deg^H(f)\Delta(P).
\end{equation}
Since $A$ is generated by $A_1$, we get an embedding 
\begin{equation}
\End(A)\subseteq {\rm Hom}_{\bfk}(A_1,A)=\cup_{d\geq 1}{\rm Hom}_{\bfk}(A_1,A_d).
\end{equation}
Set $\End(A)_d=\End(A)\cap {\rm Hom}_{\bfk}(A_1,A_d)$. For any automorphism $f\colon X\to X$, $\deg^H(f)\leq d$ if and only if $f\in \End(A)_d$.
By Riemann-Roch theorem, there exists $\gamma>0$ such that 
$\dim A_n\leq \gamma n^k$, and this gives the upper bound 
\begin{equation}\label{eq:RR}
\dim \End(A)_d\leq  {\rm Hom}_{\bfk}(A_1,A_d)\leq (\gamma d^k)\dim A_1.
\end{equation}

The following proposition, 
proved in the Appendix, shows that this new degree $\deg^H(f)$ is  equivalent to the degree $\deg_H(f)$ introduced in Section~\ref{par:Intro}.

\begin{pro}\label{pro:eq-degrees} For every automorphism $f\in \Aut(X)$ we have
\[
\frac{1}{k}\deg^H(f)\leq \frac{1}{(H^k)}\deg_H(f) \leq \deg^H(f).
\]
\end{pro}

\subsubsection{Proof of Theorem~\ref{thm:affbound}}\label{par:proof-Urech-new}
By Proposition~\ref{pro:eq-degrees},  the initial notion of degree can be replaced by $\deg^H$.
Let $\gamma$ be as in Equation~\eqref{eq:RR}. Set $\ell=(\gamma d^k)\dim A_1+1$, and assume that  $\deg^H(f^{n_i})\leq d$ 
for some  sequence  of positive integers $n_1<n_2<\ldots < n_{\ell}$. 
Each $(f^{*})^{{n_i}}$ is in $\End(A)_d$ and, because $\ell > \dim \End(A)_d$, there is a non-trivial linear relation between the
$(f^{*})^{n_i}$ in the vector space $\End(A)_d$:
\begin{equation}
(f^{*})^{n}=\sum_{m=1}^{n-1} a_m \, (f^{*})^m
\end{equation}
for some integer $n\leq n_\ell$ and some coefficients $a_m\in \bfk$. 
Then, the subalgebra $\bfk[f^*]\subseteq \End(A)$ is of finite dimension and 
 $\bfk[f^*]\subseteq E_B$ for some $B\geq 0$. 
This shows that the sequence $(\deg^H(f^N))_{N\geq 0}$ is
bounded.

Thus, if we set $\alpha=\gamma\dim A_1$, and if the sequence $(\deg^H(f^n))$ is not bounded, we obtain  
$\#\{n\geq 0|\,\, \deg^H(f^n)\leq d\}\leq \alpha d^k.
$
This proves the first assertion of the theorem; the second follows easily.

\subsection{The $p$-adic argument}  Let us  give another proof of Theorem~A when ${\mathrm{char}}(\bfk)=0$, which will be 
generalized in \S~\ref{par:BirationalTransformations}  for birational transformations.

\subsubsection{Tate diffeomorphisms}\label{par:tate-diffeomorphisms}
Let $p$ be a prime number. Let $K$ be a field of characteristic $0$ which is complete with respect to an absolute
value $\vert \cdot \vert$ satisfying $\vert p\vert = 1/p$; such an absolute value is automatically ultrametric 
(see \cite{Koblitz:book}, Ex. 2 and 3, Chap. I.2). 
Let $R=\{x\in K; \; \vert x\vert \leq 1\}$ be the valuation ring of $K$; in the vector space $K^k$, the unit {\bf{polydisk}} is  the subset~$\U=R^k$.

Fix a positive integer $k$, and consider the ring $R[\ux]=R[\ux_1, ..., \ux_k]$ of polynomial
functions in $k$ variables with coefficients in $R$. For $f$ in $R[\ux]$, define the norm
$\parallel f\parallel$ to be the supremum of the absolute values of the coefficients of $f$: 
\begin{equation}\label{eq:norm}
\parallel f \parallel=\sup_I\vert a_I\vert
\end{equation}
where 
$f=\sum_{I=(i_1, \ldots, i_k)} a_I\ux^I$.
By definition, the {\bf{Tate algebra}} $R\langle \ux\rangle$ is the completion of $R[\ux]$
with respect to this norm.
It coincides with the set of formal power series $f=\sum_I a_I \ux^I$  
converging (absolutely) on the closed polydisk $R^k$. 
Moreover, the absolute convergence is equivalent to $\vert a_I\vert \to 0$ 
as ${\mathrm{length}}(I) \to \infty$. Every element $g$ in $R\langle \ux\rangle^k$ determines a {\bf{Tate analytic}} map $g\colon\U\to \U$.

For $f$ and $g$ in $R\langle \ux\rangle$ and $c$ in $\R_+$, the notation 
$f\in p^cR\langle \ux \rangle$ 
means $\parallel f\parallel \leq \vert p\vert^c$ and the notation 
$f\equiv g \mod (p^c)$
means  $\parallel f-g\parallel \leq \vert p\vert^c$; we then extend such notations  component-wise 
to $(R\langle \ux\rangle )^m$ for all $m\geq 1$. 

For indeterminates $\ux=(\ux_1, \ldots, \ux_k)$ and $\uy=(\uy_1, \ldots, \uy_m)$, 
the composition $R\langle \uy\rangle\times R\langle \ux\rangle^m\to R\langle \ux\rangle$ is well defined, and  
coordinatewise we obtain
\begin{equation}
R\langle \uy\rangle^n\times R\langle \ux\rangle^m\to R\langle \ux\rangle^n.
\end{equation}
When $m=n=k$, we get a semigroup $R\langle \ux\rangle^k$. The group of (Tate) {\bf{analytic diffeomorphisms}}
of $\U$ is the group of invertible elements in this semigroup; we denote it by $\Diff^{an}(\U)$. Elements of $\Diff^{an}(\U)$
are bijective transformations $f  \colon     \U\to \U$ given by
$ f( \ux) = (f_1, \ldots, f_k)( \ux)$ where each $f_i$ is in $R\langle \ux\rangle$ with an inverse $f^{-1}\colon \U\to \U$ 
that is also defined by power series in the Tate algebra.  

The following result is due to Jason Bell and Bjorn Poonen (see \cite{Bell:2006, Poonen:2014}).

\begin{thm}\label{thm:Bell-Poonen}
Let $f$ be an element of $R\langle \ux \rangle^k$ with $f\equiv {\mathrm{id}}\; \mod (p^c)$
for some real number $c>1/(p-1)$. Then $f$ is a Tate diffeomorphism of $\U=R^k$ and there exists a unique Tate analytic
map $\Phi\colon R\times \U\to \U$
such that
\begin{enumerate}
\item $\Phi(n,\ux)=f^n(\ux)$ for all $n\in \Z$;
\item $\Phi(s+t,\ux)=\Phi(s, \Phi(t, \ux))$ for all $t$, $s$ in $R$.
\end{enumerate}
\end{thm}

%
%

\subsubsection{Second proof of Theorem~A}\label{par:second-proof}
Denote by $S$ the finite set of 
all the coefficients that appear in the polynomial formulas defining $f$ and $f^{-1}$. Let $R_S\subset \bfk$ be the ring 
generated by $S$ over $\Z$, and let $K_S$ be its fraction field: 
\begin{equation}
\Z\subset R_S\subset K_S\subset \bfk.
\end{equation}
Since ${\text{char}}(\bfk)=0$, there exists a prime $p>2$ such that $R_S$ embeds into $\Z_p$ (see \cite{Lech:1953}, \S 4 and 5, and \cite{Bell:2006}, Lemma 3.1). 
We apply this embedding to the coefficients of $f$ and get an automorphism of $\bbA^k_{\Q_p}$
which is defined by polynomial formulas in $\Z_p[\ux_1, \ldots, \ux_k]$; for simplicilty, 
we keep the same notation $f$ for this automorphism (embedding $R_S$ in $\Z_p$ does not change the value of the degrees $\deg(f^n)$). 
Since $f$ and $f^{-1}$ are polynomial automorphisms with coefficients in $\Z_p$, they determine elements of $\Diff^{an}(\U)$, 
the group of analytic diffeomorphisms of the polydisk $\U=\Z_p^k$. 

Reducing the coefficients of $f$ and $f^{-1}$ modulo  $p^2\Z_p$, one gets two permutations of the finite set $\bbA^k(\Z_p/p^2\Z)$ (equivalently, $f$ and $f^{-1}$ permute the balls of $\U=\Z_p^k$ of radius $p^{-2}$, and these balls are parametrized by $\bbA^k(\Z_p/p^2\Z)$; see~\cite{Cantat-Xie}).
Thus, there exists a positive integer $m$ such that $f^{m}(0)\equiv 0 \mod (p^2)$. Taking some further iterate, we may also assume that 
the differential $Df^m_0$ satisfies $Df^{m}_0\equiv \Id \mod(p)$. 
We fix such an integer $m$  and replace $f$ by $f^m$. 
The following  lemma follows from the submultiplicativity of degrees (see Equation~\eqref{eq:sub-multi} in Section~\ref{par:Intro}). It shows that replacing $f$ by $f^m$ is harmless
if one wants to bound the degrees of the iterates of $f$.

\begin{lem}\label{lem:f-to-fn}
If the sequence $\deg(f^{mn})$ is bounded for some $m>0$, then the sequence $\deg(f^n)$
is bounded too. 
\end{lem}

Denote by $\bfx=(\bfx_1, \ldots, \bfx_k)$ the coordinate system of $\bbA^k$, and by 
$m_p$ the multiplication by $p$: $m_p(\bfx)=p\bfx$.
Change $f$ into $g:=m_p^{-1}\circ f \circ m_p$; then $g\equiv \Id \mod(p)$ in the sense of Section~\ref{par:tate-diffeomorphisms}. 
Since $p\geq 3$, Theorem~\ref{thm:Bell-Poonen} gives 
a Tate analytic flow $\Phi\colon \Z_p\times \bbA^k(\Z_p)\to \bbA^k(\Z_p)$ which extends the action of $g$: 
$\Phi(n,\bfx)=g^n(\bfx)$ for every integer $n\in \Z$. Since $\Phi$ is analytic, one can write 
\begin{equation}
\Phi(\bft,\bfx)=\sum_J A_J(\bft) \bfx^J
\end{equation}
where $J$ runs over all multi-indices $(j_1, \ldots , j_k)\in (\Z_{\geq 0})^k$ and each $A_J$ defines a 
$p$-adic analytic curve $\Z_p\to \bbA^k(\Q_p)$. By submultiplicativity
of the degrees,  there is a constant $C>0$ such that $\deg(g^{n_i})\leq C B^m$. 
Thus, we obtain 
$A_J(n_i)= 0$
for all indices $i$ and all multi-indices $J$ of length $\vert J\vert > CB^m$. The $A_J$ being
analytic functions of $t\in \Z_p$, the principle of isolated zeros implies that 
\begin{equation}
A_J = 0 \, \; {\text{in}}\, \;  \Z_p\langle t\rangle, \; \,  \forall J \; {\text{with}} \; \vert J\vert > CB^m.
\end{equation}
Thus, $\Phi(t,\ux)$ is a polynomial automorphism of degree $\leq CB^m$ for all $t\in \Z_p$, and 
$g^n(\ux)=\Phi(n,\ux)$ has degree at most $CB^m$ for all $n$. By Lemma~\ref{lem:f-to-fn}, 
this proves that $\deg(f^n)$
is a bounded sequence.

\section{Birational transformations}~\label{par:BirationalTransformations}

\vspace{0.1cm}

\begin{thm-B}
Let $\bfk$ be a field of characteristic $0$. Let $X$ be a projective variety and  $f\colon X\dasharrow X$
be a birational transformation of $X$, both defined over~$\bfk$. If the sequence $(\deg(f^n))_{n\geq 0}$
 is not bounded, then it goes to $+\infty$ with~$n$: 
 \[
 \liminf_{n\to +\infty}\deg(f^n)=+\infty.
 \]
\end{thm-B}

\vspace{-0.5cm}

This extends Theorem~A  to birational transformations. With a theorem of Weil, we get: {\sl{if $f$ is a birational transformation of the
projective variety $X$, over an algebraically closed field of characteristic $0$, and if the degrees of its iterates 
are bounded along an infinite
subsequence $f^{n_i}$, then there exist a birational map $\psi\colon Y\dasharrow X$ and an integer $m>0$ 
 such that $f_Y:=\psi^{-1}\circ f\circ \psi$ is in $\Aut(Y)$, and $f_Y^m$ is in the connected component $\Aut(Y)^0$}} (see~\cite{Cantat:Compositio} and references therein).

Urech's argument does not apply to this context; the basic obstruction is that rational transformations of $\bbA^k_\bfk$ of degree $\leq B$ generate 
an infinite dimensional $\bfk$-vector space for every $B\geq 1$ (the maps $z\in \bbA^1_\bfk \mapsto (z-a)^{-1}$ 
with $a\in \bfk$ are linearly independent); looking back at the proof  in Section~\ref{par:proof-Urech-new}, the problem is that the field of rational functions on an affine variety $X$ is not finitely generated as a $\bfk$-algebra. We shall adapt the $p$-adic method described in Section~\ref{par:second-proof}. 
In what follows, $f$ and $X$ are as in Theorem~B; we assume, without loss of generality,  that $\bfk=\C$ and  $X$ is smooth. 
We suppose that there is an infinite sequence of integers $n_1 <  \ldots < n_j <\ldots$ and a number $B$ such that $\deg(f^{n_j})\leq B$ 
for all $j$. We fix a finite subset $S\subset \C$ such that $X$, $f$ and $f^{-1}$ are defined by equations and formulas 
with coefficients in $S$, and we embed the ring $R_S\subset \C$ generated by $S$ in some $\Z_p$, for  some prime number $p>2$. 
According to \cite[Section~3]{Cantat-Xie}, we may assume that $X$ and $f$ have good reduction modulo $p$.

\subsection{The Hrushovski's theorem and $p$-adic polydisks}

According to a theorem of Hrushovski (see~\cite{Hrushovski}), there is a periodic point $z_0$ of $f$ in $X(\bfF)$ for some finite field extension $\bfF$
of the residue field $\bfF_p$, the orbit of which does not intersect  the indeterminacy points
of $f$ and $f^{-1}$. If $\ell$ is the period of $z_0$, then $f^\ell(z_0)=z_0$ and  $Df^\ell_{z_0}$ is an element of 
the finite group $\GL((TX_{\bfF_q})_{z_0})\simeq \GL(k,\bfF_q)$. Thus, there is an integer $m>0$ such that $f^m(z_0)=z_0$
and $Df^m_{z_0}=\Id$.

Replace $f$ by its iterate $g=f^m$. Then, $g$ fixes $z_0$ in $X(\bfF)$, $g$ is an isomorphism in a
neighborhood of $z_0$, and $Dg_{z_0}=\Id$. According to~\cite{Bell-Ghioca-Tucker:2010} and \cite[Section~3]{Cantat-Xie}, this implies that there is 
\begin{itemize}
\item a finite extension $K$ of $\Q_p$, 
with  valuation ring $R\subset K$; 
\item a point $z$ in 
$X(K)$ and a polydisk $ \V_z  \simeq R^k\subset X(K)$ which is $g$-invariant and such that $g|_{V_z}\equiv \Id \mod(p)$
(in the coordinate system $(\ux_1,\ldots, \ux_k)$ of the polydisk). 
\end{itemize}
When the point $z_0$ is in $X(\bfF_p)$ and is the reduction of a point $z\in X(\Z_p)$, the polydisk $\V_z$ is the
set of points $w\in X(\Z_p)$ with $\vert z-w\vert <1$; one identifies this polydisk to $\U=(\Z_p)^k$ via some $p$-adic
analytic diffeomorphism $\varphi\colon \U\to \V_z$; changing $\varphi$ into $\varphi\circ m_p$ if necessary, we
obtain $g_{V_z}\equiv \Id \mod(p)$ (see Section~\ref{par:second-proof} and \cite{Cantat-Xie}, Section~3.2.1). In full generality, a finite
extension $K$ of $\Q_p$ is needed because $z_0$ is a point in $X(\bfF)$ for some extension $\bfF$ of $\bfF_p$. 

\subsection{Controling the degrees}
As in Section~\ref{par:tate-diffeomorphisms}, denote by $\U$ the polydisk $R^k\simeq \V_z$; thus, $\U$ 
is viewed as the polydisk $R^k$ and also as a subset of $X(K)$.
Applying Theorem~\ref{thm:Bell-Poonen} to $g$, we obtain a $p$-adic analytic flow
\begin{equation}
\Phi\colon R\times \U\to \U, \quad (t,\bfx)\mapsto \Phi(t,\bfx)
\end{equation}
such that  $\Phi(n,\bfx)=g^n(\bfx)$ for every integer $n$. 
In other words, the action of $g$ on $\U$ extends to an analytic action of the additive
compact group $(R,+)$.

Let $\pi_1\colon X\times X\to X$
denote the projection onto the first factor. Denote by $\Bir_D(X)$ the set of birational transformations of 
$X$ of degree $D$; once birational transformations are identified to their graphs,
this set becomes naturally a finite union of irreducible, locally closed algebraic subsets in the Hilbert scheme of $X\times X$
(see~\cite{Cantat:Compositio}, Section~2.2, and references therein).
Taking a subsequence, there is a positive integer $D$,  an irreducible component $B_{D}$ of $\Bir_D(X)$, 
and a strictly increasing, infinite sequence of integers $(n_j)$ such that 
\begin{equation}
g^{n_j}\in B_{D}
\end{equation}
for all $j$. Denote by $\overline{B_D}$ the Zariski closure of $B_D$ in the Hilbert scheme of $X\times X$. 
To every element $h\in \overline{B_D}$ corresponds a unique algebraic subset $\Gr_h$ of $X\times X$
(the graph of $h$, when $h$ is in $B_D$).
Our goal is to show that, for every $t\in R$,
 the graph of $\Phi(t,\cdot)$ is the intersection $\Gr_{h_t}\cap \U^2$ for some  element $h_t\in \overline{B_D}$; this
will conclude the proof because $g^n(\ux)=\Phi(n,\ux)$ for all $n\geq 0$. 

We start with a simple remark, which we encapsulate in the following lemma. 

\begin{lem}
There is a finite subset $E\subset \U\subset X(K)$ with the following property. Given any subset 
$\tilde{E}$ of $\, \U\times \U$ with $\pi_1(\tilde{E})=E$, there is at most one element $h\in \overline{B_D} $
such  that $\tilde{E}\subset \Gr_h$.
\end{lem}

Fix such a set $E$, and order it to get a finite list $E=(x_1, \ldots, x_{\ell_0})$ of elements of $\U$.
Let 
$E'=(x_1, \ldots, x_{\ell_0}, x_{\ell_0+1}, \ldots, x_\ell)$  be any list of elements of $\U$ which extends $E$. 
 For every element $h$ in ${\overline{B_D}}$, 
the variety $\Gr_h$ determines a correspondance 
$\Gr_h\subset X\times X$. 
The subset of elements $(h,(x_i,y_i)_{1\leq i\leq \ell})$ in ${\overline{B_D}}\times (X\times X)^\ell$ 
defined by the incidence relation
\begin{equation}
(x_i,y_i)\in \Gr_h
\end{equation}
for every $1\leq i\leq \ell$ is an algebraic subset of ${\overline{B_D}}\times (X\times X)^\ell$.
Add one constraint, namely that the first projection $(x_i)_{1\leq i\leq \ell}$ coincides with $E'$, and
project the resulting subset on $(X\times X)^\ell$: we get a subset $G(E')$ of $(X\times X)^\ell$.
Then, define a $p$-adic analytic curve  $\Lambda\colon R\to (X\times X)^\ell$ by
\begin{equation}
\Lambda(t)=(x_i,\Phi(t,x_i))_{1\leq i\leq \ell}.
\end{equation}
If $t=n_j$,  $g^{n_j}$ is an element of $B_D$ and $\Lambda(n_j)$ is contained in the
graph of $g^{n_j}$; hence,
$\Lambda(n_j)$ is an element of $G(E')$.
By the principle of isolated zeros,  the analytic curve $t\mapsto \Lambda(t)\subset (X\times X)^\ell$ is contained 
in $G(E')$ for all $t\in R$. 
Thus, for every $t$ there is an element $h_t\in {\overline{B_D}}$  
such that $\Lambda(t)$ is contained in  the subset $\Gr_{h_t}^\ell$ of $(X\times X)^\ell$. From the choice of $E$ and the inclusion $E\subset E'$, 
we know that $h_t$ does not
depend on $E'$. Thus, the graph of $\Phi(t,\cdot)$ coincides with the intersection of $\Gr_{h_t}$ 
with $\U\times \U$. This implies that the graph of $g^n(\cdot)=\Phi(n,\cdot)$ coincides with $\Gr_{h_n}$, and that the degree of $g^n$
is at most $D$ for all values of $n$. 

\section{Lower bounds on degree growth}\label{par:Ackermann}

We now prove that the growth of $(\deg(f^n))$ can not be arbitrarily slow unless
 $(\deg(f^n))$ is bounded. For simplicity, we focus on birational transformations of the projective space; there is 
no restriction on the characteristic of~$\bfk$. 
 
\subsection{A family of integer sequences}\label{par:Notations} 
Fix two positive integers $k$ and $d$; $k$ will be the dimension of $\bbP^k_\bfk$, and
$d$ will be the degree of $f\colon \bbP^k\dasharrow \bbP^k$. Set
\begin{equation}
m=(d-1)(k+1).
\end{equation}
Then, consider an auxiliary integer $D\geq 1$, which will play the role of the degree of an
effective divisor in the next paragraphs, and  define 
\begin{equation}
q=(dD+1)^m.
\end{equation}
Thus, $q$ depends on $k$, $d$ and $D$ because $m$ depends on $k$ and $d$. Then, set
\begin{equation}
a_0=\left(\begin{array}{c} k+D \\ k\end{array}\right)-1, \quad b_0= 1, \quad c_0= D+1.
\end{equation}
Starting from the triple $(a_0,b_0,c_0)$, we define a sequence $((a_j, b_j, c_j))_{j\geq 0}$
inductively by
\begin{equation}
(a_{j+1}, b_{j+1}, c_{j+1})=(a_j, b_j-1, q c_j^2) 
\end{equation}
if $b_j\geq 2$, and by
\begin{equation}
(a_{j+1}, b_{j+1}, c_{j+1})=(a_j-1, qc_j^2 , q c_j^2)= (a_j-1, c_{j+1}, c_{j+1}) 
\end{equation}
if $b_j=1$. By construction,  $(a_1, b_1, c_1) = (a_0-1, q c_0^2, q c_0^2)$.

Define $\Phi\colon \Z^+ \to \Z^+$ by 
\begin{equation}
\Phi(c) = q c^2.
\end{equation}

\begin{lem}
Define the sequence of integers $(F_i)_{i\geq 1}$ recursively by $F_1=q(D+1)^2$ and 
$ F_{i+1}= \Phi^{F_i}(F_i) $
for $i\geq 1$ (where $\Phi^{F_i}$ is the $F_i$-iterate of $\Phi$). Then 
\[
(a_{1+F_1+\cdots + F_i}, \,  b_{1+F_1+\cdots + F_i}, \, c_{1+F_1+\cdots + F_i})= (a_0-i-1, \, F_{i+1}, \, F_{i+1}).
\]
\end{lem}

The proof is straightforward. 
Now, define $S\colon \Z^+\to \Z^+$ as the sum
\begin{equation}\label{eq:S}
S(j)=1+F_1+F_2+\cdots + F_j
\end{equation}
for all $j\geq 1$; it is increasing  and goes to $+\infty$ extremely fast with~$j$. 
Then, set
\begin{equation}
\chi_{d,k}(n)=\max\left\{ D\geq 0\; \vert \; \; S(\left(\begin{array}{c} k+D \\ k\end{array}\right)-2) < n \right\}.
\end{equation}

\begin{lem}
The function $\chi_{d,k}\colon \Z^+\to \Z^+$ is  non-decreasing and  goes to $+\infty$ with $n$. 
\end{lem}

\begin{rem} {\sl{The function $S$ is primitive recursive}}  (see~\cite{Davis-Weyuker}, Chapters~3 and~13). In other words, $S$ is obtained from the
basic functions (the zero function, the successor $s(x)=x+1$, and the projections $(x_i)_{1\leq i\leq m}\to x_i$)
by a finite sequence of 
compositions and recursions. Equivalently,  there is a program 
computing $S$, all of whose instructions are limited to (1) the zero initialization $V \leftarrow 0$, 
(2) the increment $V\leftarrow V+1$, (3) the assignement $V \leftarrow V'$, and (4) loops of definite
length. Writing such a program is an easy exercise. 
Now, consider the diagonal Ackermann function $A(n)$ (see~\cite{Davis-Weyuker}, Section~13.3). It grows
asymptotically faster than any primitive recursive function; hence, the inverse of the Ackermann diagonal
function 
$\alpha(n)=\max\{D\geq 0 \; \vert \; {\text{Ack}}(D)\leq n\}$ is, asymptotically, a lower bound for  $\chi_{d,k}(n)$. Showing that 
$\chi_{d,k}$ is in the ${\mathcal{L}}_6$ hierarchy of \cite{Davis-Weyuker}, Chapter~13, one gets 
an asymptotic lower bound by the inverse of the function $f_7$ of \cite{Davis-Weyuker}, independent of the values
of $d$ and $k$. 
\end{rem}

\subsection{Statement of the lower bound} 

We can now state the result that will be proved in the next paragraphs. \\

\begin{thm-C}
Let $f$ be a birational transformation of the complex projective space $\bbP^k_\bfk$ of degree $d$. If the
sequence $(\max_{0\leq j\leq n}(\deg(f^j)))_{n\geq 0}$ is unbounded, then it is bounded from below by the sequence of integers
$(\chi_{d,k}(n))_{n\geq 0}$.
\end{thm-C}

\vspace{-0.1cm}

\begin{rem}
There are infinitely, but only countably many sequences of degrees $(\deg(f^n))_{n\geq 0}$ (see~\cite{Bonifant2000, Urech}).
Consider the countably many sequences 
\begin{equation}
\left( \max_{0\leq j\leq n}(\deg(f^j))\right)_{n\geq 0}
\end{equation}
 restricted to the
family of birational maps for which $(\deg(f^n))$ is unbounded. We get a countable 
family of {\sl{non-decreasing, unbounded sequences of integers}}. 
Let $(u_i)_{i\in \Z_{\geq 0}}$ be any countable family of such sequences of integers
$(u_i(n))$. Define 
$w(n)$ as follows. First, set $v_j=\min\{u_0, u_1, \ldots, u_j\}$; this defines a new family of sequences, with 
the same limit $+\infty$, but now $v_j(n)\geq v_{j+1}(n)$ for every pair $(j,n)$. Then, set $m_0=0$, and
define $m_{n+1}$ recursively to be the first positive integer such that $v_{n+1}(m_{n+1})\geq v_n(m_n)+1$.  We have $m_{n+1}\geq m_n+1$ for all $n\in \Z_{\geq 0}.$
Set  $w(n):=v_{r_n}(m_{r_n})$ where $r_n$ is the unique non-negative integer satisfying $m_{r_n}\leq n\leq m_{r_n+1}-1$.
 By construction, 
$w(n)$ goes to $+\infty$ with $n$ and $u_i(n)$ is {\sl{asymptotically}} bounded from below by $w(n)$. 

In Theorem~C, the result is more explicit. Firstly, the lower bound is explicitely given by the sequence $(\chi_{d,k}(n))_{n\geq 0}$. Secondly, 
the lower bound is not asymptotic: it works for every value of $n$. In particular, if $\deg(f^j)< \chi_{d,k}(n)$
for $0\leq j\leq n$ and $\deg(f)=d$, then the sequence $(\deg(f^n))$ is bounded. 
\end{rem}

\subsection{Divisors and strict transforms}
To prove Theorem~C, we consider the action of $f$ by strict transform on effective divisors. 
As above, $d=\deg(f)$   and $m=(d-1)(k+1)$ (see Section~\ref{par:Notations}).

\subsubsection{Exceptional locus}
Let $X$ be a smooth
projective variety and $\pi_1$ and $\pi_2\colon X\to \bbP^k$ be two birational morphisms
such that $f=\pi_2\circ\pi_1^{-1}$; then, consider the exceptional locus $\Exc(\pi_2)\subset X$, 
project it by $\pi_1$ into $\bbP^k$, and list its irreducible components of codimension $1$: 
we obtain a finite number 
\begin{equation}
E_1,\;  \ldots, \; E_{m(f)}
\end{equation}
of irreducible hypersurfaces, contained in the zero locus of the jacobian determinant of $f$. 
Since this critical locus has degree $m$, we obtain: 
\begin{equation}
m(f)\leq m, \quad {\text{and}} \; \; \deg(E_i)\leq m\quad (\forall i\geq 1).
\end{equation}

\subsubsection{Effective divisors}

Denote by $M$ the semigroup of effective divisors of~$\bbP_\bfk^k$.
There is a partial ordering $\leq $ on $M$, which is defined by 
$E\leq E'$ if and only if the divisor $E'-E$ is effective. 

We denote by $\deg\colon M\to \Z_{\geq 0}$ the degree function.
For every degree $D\geq 0$, we denote by $M_D$ the set $\bbP(H^0(\bbP^k_\bfk, {\mathcal{O}}_{\bbP^k_\bfk}(D)))$ 
of effective divisors of degree $D$; thus, $M$ is the disjoint union of all the $M_D$, and each of these
components will be endowed with the Zariski topology of $\bbP(H^0(\bbP^k_\bfk, {\mathcal{O}}_{\bbP^k_\bfk}(D)))$. 
The dimension of $M_D$ is equal to the integer $a_0=a_0(D,k)$ from Section~\ref{par:Notations}:
\begin{equation}
\dim(M_D)= \left(\begin{array}{c} k+D \\ k\end{array}\right)-1.
\end{equation}

Let $G\subset M$ be the semigroup generated by the $E_i$: 
\begin{equation}
G=\bigoplus_{i=1}^{m(f)}\Z_{\geq 0} E_i.
\end{equation}
The elements of $G$ are the effective divisors which are supported by the exceptional locus of $f$.
For every  $E\in G$, there is a translation operator $T_E\colon M\to M$, defined
by $T_E\colon E'\mapsto E+E'$; it restricts to  a linear
projective embedding of the projective space $M_D$ into the projective space $M_{D+\deg(E)}$.
We define
\begin{equation}
M_D^\circ= M_D\setminus \bigcup_{E\in G\setminus\{0\},  \deg(E)\leq D} T_E(M_{D-\deg(E)}).
\end{equation}
Thus, $M_D^\circ $ is the complement in $M_D$ of finitely many proper linear
projective subspaces. Also, $M_0^\circ=M_0$ is a point and $M_1^\circ$ is obtained from $M_1=(\bbP^k_\bfk)^\vee$ 
by removing finitely many points, corresponding to the $E_i$ of degree~$1$ (the hyperplanes
contracted by $f$). Set 
$M^\circ =\bigcup_{D\geq 0} M_D^\circ.$ This is the set of effective divisors without any component in the exceptional locus of $f$.
The inclusion of $M^\circ$ in $M$ will be denoted by $\iota\colon M^\circ \to M$.
There is a natural projection $\pi_G\colon M\to G$; namely,  $\pi_G(E)$ is the
maximal element such that $E-\pi_G(E)$ is effective. We denote by 
$\pi_\circ\colon M\to M^\circ$ the projection $\pi_\circ = \Id-\pi_G$; this homomorphism 
removes the part of an effective divisor $E$ which is supported on the exceptional locus of $f$.

\begin{rem}
The restriction of the map $\pi_\circ$ to the projective space $M_D$
 is piecewise linear, in the following sense. Consider the subsets 
$U_{E,D}$ of $M_D$ which are defined for every $E\in G$ with $\deg(E)\leq D$ by
\[
U_{E,D}= T_E(M_{D-\deg(E)})\setminus \bigcup_{E'>E, \; E'\in G,  \; \deg(E')\leq D} T_{E'}(M_{D-\deg(E')}).
\] 
They define a stratification of $M_D$ by 
(open subsets of) linear subspaces, and $\pi_\circ$ coincides with the linear map inverse of $T_E$  on each $U_{E,D}$. 
Moreover, $\pi_{\circ}(Z)$ is closed for any closed subset $Z\subseteq M_D$.
\end{rem}

We say that a scheme theoritic point $x\in M$ (resp. $M^{\circ}$) is {\bf{irreducible}}
if the divisor of $\bbP^k$ corresponding to $x$ is irreducible. In other words, $x$ is irreducible, 
if a general closed point $y\in {\overline {\{x\}}}\subseteq M$ is irreducible.

\subsubsection{Strict transform}

First, we consider the total transform $f^*\colon M\to M$, which is defined by 
$f^*(E)=(\pi_1)_*\pi_2^*(E)$ for every divisor $E\in M$.
This  is a homomorphism of semigroups; it is injective on non-closed irreducible points.
Let $[x_0,\ldots, x_k]$ be homogeneous coordinates on $\bbP^k$. 
If $E$ is defined by the homogeneous equation $P=0$, 
then $f^*(E)$ is defined by $P\circ f=0$; thus, $f^*$ induces  a linear projective embedding of $M_D$ into $M_{dD}$ for every $D$. 
 
 Then, we denote by $f^\circ\colon M^\circ \to M^\circ$ the strict transform. It is defined by 
\begin{equation}
f^\circ(E)=(\pi_\circ\circ f^*\circ \iota)(E).
\end{equation}
This is a homomorphism of semigroups. 
If $x\in M$ is an irreducible point, its total transform $f^*(x)$ is not necessarily irreducible, but $f^{\circ}(x)$ is irreducible.

In general, $(f^\circ)^n \neq (f^n)^\circ$, but for non-closed irreducible point $x\in M$, we have $(f^\circ)^n(x) = (f^n)^\circ(x)$ for $n\geq 0$.
Indeed, a non-closed  irreducible point $x\in M$ can be viewed as an irreducible hypersurface on $X$ which is defined 
over some transcendental extension of $\bfk$, but not over $\bfk$.
Then $f^{\circ}(x)$ is the unique irreducible component $E$ of $f^*(x)$, on which $f|_E$ is birational to its image. 
(Note that when $\bfk$ is uncountable, one can also work with very general points of $M_D$ for every $D\geq 1$,
instead of irreducible, non-closed points). 

\subsection{Proof of Theorem~C}

Let $\eta$ be the generic point of $M_1^\circ $ ($\eta$ corresponds to a generic hyperplane
of $\bbP^k_\bfk$). Note that $\eta$ is non-closed and irreducible. 
The degree of $f^*(\eta)$ is equal to the degree of $f$, and since $\eta$ is
generic, $f^*(\eta)$  coincides with $f^\circ (\eta)$. Thus, $\deg(f)=\deg(f^\circ(\eta))$ and more generally
\begin{equation}\label{eq:deg-eta}
\deg(f^n)= \deg((f^\circ)^n \eta) \quad (\forall n\geq 1).
\end{equation}

Fix an integer $D\geq 0$. Write $M^\circ_{\leq D}$ for the disjoint union of the $M^\circ_{D'}$ with $D'\leq D$, and
define recursively $Z_D(0)=M^\circ_{\leq D}$ and 
\begin{equation}
Z_D(i+1)=\{E\in Z_D(i)\; \vert \; f^\circ(E)\in Z_D(i)\}
\end{equation}
for $i\geq 0$. A divisor  $E\in M^\circ_{\leq D}$ is in $Z_D(i)$ if its strict transform $f^\circ(E)$ is of degree $\leq D$, and
$f^\circ(f^\circ(E))$ is also of degree $\leq D$, up to $(f^\circ)^i(E)$ which is also of degree
at most $D$. 

Let us describe $Z_D(i+1)$ more precisely. For each $i$, and each $E\in G$ of degree $\deg(E)\leq dD$
consider the subset $T_E(\overline{\iota(Z_D(i))})\cap M_{dD}$; this is a subset of $M_{dD}$ which is 
made of divisors $W$ such that $\pi_\circ(W)$ is contained in $Z_D(i)$, and the union of all these subsets
when $E$ varies is exactly the set of points $W$ in $M_{dD}$ with a projection $\pi_\circ(W)$ in $Z_D(i)$. 
Thus, we consider 
\begin{equation}
(f^*)^{-1}(T_E(\overline{\iota(Z_D(i))}))=\{ V\in M_{\leq D}\; \vert \; f^*(V)\in  T_E(\overline{\iota(Z_D(i))}) \}.
\end{equation}
These sets are closed subsets of $M_{\leq D}$, and 
\begin{equation}
Z_D(i+1)=Z_D(i)\bigcap         \bigcup_{E\in G, \deg(E)\leq dD}   \pi_\circ\left( (f^*)^{-1}(T_E( \overline{\iota(Z_D(i))})      \right).
\end{equation}
Since $Z_D(0)$ is closed in $M_{\leq D}^{\circ}$ and $\pi_{\circ}$ is closed on $M_{\leq D}$, by induction,  $Z_D(i)$ is closed for all $i\geq 0.$ 
The subsets $Z_D(i)$ form a decreasing sequence of Zariski closed subsets (in the disjoint
union $M^\circ_{\leq D}$ of the $M^\circ_{D'}$, $D'\leq D$). The strict transform $f^\circ$ maps $Z_D(i+1)$ into $Z_D(i)$.
By Noetherianity, there exists a minimal integer $\ell(D)\geq 0$ such that 
\begin{equation}
Z_D(\ell(D))=\bigcap_{i\geq 0}Z_D(i);
\end{equation}
we denote this subset by $Z_D(\infty)=Z_D(\ell(D))$. By construction, $Z_D(\infty)$ is stable under the operator $f^\circ$; 
more precisely, $f^\circ(Z_D(\infty))=Z_D(\infty)=(f^\circ)^{-1}(Z_D(\infty))$.

Let $\tau\colon \Z_{\geq 0}\to \Z_{\geq 0}$ be a lower bound for the inverse function of $\ell$: 
\begin{equation}
\ell(\tau(n))\leq n \quad (\forall n\geq 0). 
\end{equation}
Assume that 
$\max\{\deg(f^m)\; \vert\; 0\leq m \leq n_0\}\leq \tau(n_0)$ for some $n_0\geq 1$. Then
$\deg((f^\circ)^i(\eta))\leq \tau(n_0)$ for every integer $i$ between 
$0$ and $n_0$; this implies that $\eta$ is in 
the set $Z_{\tau(n_0)}(\ell(\tau(n_0)))=Z_{\tau(n_0)}(\infty)$, so that the degree of
 $(f^\circ)^m(\eta)$ is bounded from above by $\tau(n_0)$ for all $m\geq 0$. From Equation~\eqref{eq:deg-eta}
 we deduce that the sequence $(\deg(f^m))_{m\geq 0}$ is bounded. This proves the following lemma. 

\begin{lem}
Let $\tau$ be a lower bound for the inverse function of~$\ell$.
If  
\[
\max\{\deg(f^m)\; \vert\; 0\leq m \leq n_0\}\leq \tau(n_0)
\]
for some $n_0\geq 1$, then the sequence $(\deg(f^n))_{n\geq 0}$ is bounded by $\tau(n_0)$.
\end{lem}

So, to conclude, we need to compare $\ell\colon \Z_{\geq 0}\to \Z^+$ to the 
function $S\colon \Z_{\geq 0}\to \Z^+$ of paragraph~\ref{par:Notations} (recall that $S$ depends 
on the parameters $k=\dim(\bbP^k_\bfk)$ and $d=\deg(f)$ and that $\ell$ depends on $f$).
Now, write $Z_D'(i)=Z_D(i)\setminus Z_D(\infty)$, and note that it is a strictly decreasing sequence of open subsets of  $Z_D(i)$ with 
$Z_D'(j)=\emptyset$ for all $j\geq \ell(D)$. 
We shall say that a closed subset of $M^\circ_{\leq D}\setminus Z_D(\infty)$ for the Zariski topology is  {\bf{piecewise linear}} if
all its irreducible components are equal to the intersection of $M^\circ_{\leq D}\setminus Z_D(\infty)$ with a linear
projective subspace of some $M_{D'}$, $D'\leq D$.  We note that the intersection of two irreducible linear projective subspaces is still an irreducible linear projective subspace.

Let $\Lin(a,b,c)$ be the family of closed piecewise linear
subsets of $M^\circ_{\leq D}\setminus Z_D(\infty)$ of dimension $a$, with at most $c$ irreducible components, 
and at most $b$ irreducible components of maximal dimension $a$. 
Then,
\begin{enumerate}
 \item $Z_D'(i+1)=\{F \in Z_D'(i)\; \vert \; f^\circ(F)\in Z_D'(i)\}= \pi_\circ (f^*Z_D'(i)\bigcap \cup_ET_E(Z_D'(i)))$, where
$E$ runs over the elements of $G$ of degree $\deg(E)\leq dD$;
\item in this union, each irreducible component of $T_E(Z_D'(i))$ is piecewise linear. 
\end{enumerate}
Recall that $q=(dD+1)^m$ (see Section~\ref{par:Notations}). 
If $Z$ is any closed piecewise linear subset of $M^\circ_{\leq D}\setminus Z_D(\infty)$ that contains exactly 
$c$ irreducible components, the set 
\begin{equation*}
\begin{aligned}
\pi_\circ (f^*Z\bigcap \bigcup_{E\in G, \; \deg(E)\leq dD} T_E(Z))&=\bigcup_{E\in G, \; \deg(E)\leq dD} \pi_\circ (f^*Z\bigcap T_E(Z))\\
&=\bigcup_{E\in G, \; \deg(E)\leq dD} T_E^{-1}|_{T_E(Z)} (f^*Z\bigcap T_E(Z))
\end{aligned}
\end{equation*}
has at most 
$qc^2=(dD+1)^mc^2$ irreducible components (this is a crude estimate:  $f^*Z\bigcap T_E(Z)$ has at most $c^2$ irreducible components,  $T_E^{-1}|_{T_E(Z)}$ is injective
and the factor $(dD+1)^m$ comes from the fact that $G$ contains at most $(dD+1)^m$ elements of degree $\leq dD$). Let us
now use that the sequence $Z'_D(i)$ decreases strictly as $i$ varies from $0$ to $\ell(D)$, with $Z'_D(\ell(D))=\emptyset$.
If $0\leq i\leq \ell(D)-1$, 
and if $Z'_D(i)$ is contained in $\Lin(a,b,c)$, we obtain 
\begin{enumerate}
\item if $b\geq 2$, then $Z'_D(i+1)$ is contained in $\Lin(a,b-1,qc^2)$;
\item if $b=1$, then $Z'_D(i+1)$ is contained in $\Lin(a-1,qc^2,qc^2)$.
\end{enumerate}
This shows that 
\begin{equation}
\ell(D)\leq S(\left(\begin{array}{c} k+D \\ k\end{array}\right)-2)+1
\end{equation}
where $S$ is the function introduced in the Equation~\eqref{eq:S} of Section~\ref{par:Notations}. Since 
$\chi_{d,k}$ satisfies $\ell(\chi_{d,k}(n))\leq n$ for every $n\geq 1$, the conclusion follows.

%
%
\section{Appendix: Proof of Proposition~\ref{pro:eq-degrees}}
%
%

{\small{
We keep the notation from Section~\ref{par:degrees-affine}. Let $f$ be an automorphism of $X$. 
There exist a normal projective irreducible variety $Z$ and two birational morphisms $\pi_1\colon Z\to Y$ and $\pi_2\colon Z\to Y$ such that $\pi_1$ and $\pi_2$ are isomorphisms over $X$, and $f=\pi_2\circ \pi_1^{-1}.$

\begin{lem}\label{lemdpboundbydegree} We have $\Delta({f^*P})\leq k(H^k)^{-1}\Delta(P)\deg_H(f)$ for every $P\in A$.
\end{lem}
\begin{proof}[Proof of Lemma \ref{lemdpboundbydegree}]
By Siu's inequality (see~\cite{Lazarsfeld:2004} Theorem~2.2.15, and~\cite{Cutkosky} Theorem~1), we get
\begin{equation}
\pi_2^*H\leq \frac{k(\pi_2^*H\cdot (\pi_1^*H)^{k-1})}{((\pi_1^*H)^k)}\pi_1^*H=\frac{k\deg_H(f)}{(H^k)}\pi_1^*H.
\end{equation}
Since $(P)+\Delta(P)H\geq 0$ we have $(\pi_2^*P)+\Delta(P)\pi_2^*H\geq 0.$
It follows that 
\begin{equation}
(\pi_2^*P)+\frac{\Delta(P)k\deg_H(f)}{(H^k)}\pi_1^*H\geq 0.
\end{equation}
Since $(\pi_1)_*\circ (\pi_1)^* =\Id$ we obtain $(f^*P)+(k\Delta(P)(H^k)^{-1}\deg_H(f))H\geq 0$.
This implies   $\Delta({f^*P})\leq k(H^k)^{-1}\Delta(P) \deg_H(f)$.
\end{proof}

Lemma \ref{lemdpboundbydegree} shows that $\deg^H(f)\leq k(H^k)^{-1}\deg_H(f)$.
We now prove the reverse direction: $\deg_H(f)\leq (H^k)\deg^H(f)$.

Since $H$ is very ample,  
Bertini's theorem gives an irreducible divisor $D\in |H|$ such that 
$\pi_2(E)\not\subseteq D$ for every prime divisor $E$ of  $Z$ in $Z\setminus \pi^*_2(X)$;
hence,  $\pi_2^*D$ is equal to the strict transform $\pi_2^\circ D$.
By definition, $D=(P)+H$ for some  $P\in A_1$.  
Thus, $(\pi_1)_*\pi^*_2H$ is linearly equivalent to 
$(\pi_1)_*\pi^*_2D=(\pi_1)_*\pi_2^\circ D$, and this irreducible divisor $(\pi_1)_*\pi_2^\circ D$ is the closure $D_{f^*P}$ of $\{f^*P=0\}\subseteq X$ in Y.
Writing $(f^*P)=D_{f^*P}-F$ where $F$ is supported on $Y\setminus X$ we also get that $(\pi_1)_*\pi^*_2H$ is linearly equivalent to $F$.
Since $\Delta({f^*P})\leq \deg^H(f)\Delta(P)=\deg^H(f)$, the definition of $\Delta$ gives 
\begin{equation}
D_{f^*P}-F+\deg^H(f)H=(f^*P)+\deg^H(f)H\geq 0.
\end{equation}  
Thus, $F\leq \deg^H(f)H$ because $D_{f^*P}$ is irreducible and is not supported on $Y\setminus X$.
Altogether, this gives 
$\deg_H(f)=((\pi_1)_*\pi^*_2H\cdot H^{k-1})=(F\cdot H^{k-1})\leq \deg^H(f)(H^k).
$
}}

%
%

%
%

\bibliographystyle{plain}
 
\bibliography{references}

\def\cprime{$'$}
\begin{thebibliography}{10}

\bibitem{Bell:2006}
Jason~P. Bell.
\newblock A generalised {S}kolem-{M}ahler-{L}ech theorem for affine varieties.
\newblock {\em J. London Math. Soc. (2)}, 73(2):367--379, 2006.

\bibitem{Bell-Ghioca-Tucker:2010}
Jason~P. Bell, Dragos Ghioca, and Thomas.~J. Tucker.
\newblock The dynamical {M}ordell-{L}ang problem for \'etale maps.
\newblock {\em Amer. J. Math.}, 132(6):1655--1675, 2010.

\bibitem{Blanc-Cantat}
J\'er\'emy Blanc and Serge Cantat.
\newblock Dynamical degrees of birational transformations of projective
  surfaces.
\newblock {\em J. Amer. Math. Soc.}, 29(2):415--471, 2016.

\bibitem{Bonifant2000}
Araceli~M. Bonifant and John~Erik Forn\ae~ss.
\newblock Growth of degree for iterates of rational maps in several variables.
\newblock {\em Indiana Univ. Math. J.}, 49(2):751--778, 2000.

\bibitem{Cantat:Compositio}
Serge Cantat.
\newblock Morphisms between {C}remona groups, and characterization of rational
  varieties.
\newblock {\em Compos. Math.}, 150(7):1107--1124, 2014.

\bibitem{Pano}
Serge Cantat, Antoine Chambert-Loir, and Vincent Guedj.
\newblock {\em Quelques aspects des syst\`emes dynamiques polynomiaux},
  volume~30 of {\em Panoramas et Synth\`eses [Panoramas and Syntheses]}.
\newblock Soci\'et\'e Math\'ematique de France, Paris, 2010.

\bibitem{Cantat-Xie}
Serge Cantat and Junyi Xie.
\newblock Algebraic actions of discrete groups: the {$p$}-adic method.
\newblock {\em Acta Math.}, 220(2):239--295, 2018.

\bibitem{Cutkosky}
Steven~Dale Cutkosky.
\newblock Teissier's problem on inequalities of nef divisors.
\newblock {\em J. Algebra Appl.}, 14(9):1540002, 37, 2015.

\bibitem{Davis-Weyuker}
Martin~D. Davis and Elaine~J. Weyuker.
\newblock {\em Computability, complexity, and languages}.
\newblock Computer Science and Applied Mathematics. Academic Press, Inc.
  [Harcourt Brace Jovanovich, Publishers], New York, 1983.
\newblock Fundamentals of theoretical computer science.

\bibitem{Diller-Favre}
Jeffrey Diller and Charles Favre.
\newblock Dynamics of bimeromorphic maps of surfaces.
\newblock {\em Amer. J. Math.}, 123(6):1135--1169, 2001.

\bibitem{Dinh-Sibony:2005}
Tien-Cuong Dinh and Nessim Sibony.
\newblock Une borne sup\'erieure pour l'entropie topologique d'une application
  rationnelle.
\newblock {\em Ann. of Math. (2)}, 161(3):1637--1644, 2005.

\bibitem{Hrushovski}
Ehud Hrushovski.
\newblock The elementary theory of the {F}robenius automorphism.
\newblock {\em http://arxiv.org/pdf/math/0406514v1}, pages 1--135, 2004.

\bibitem{Koblitz:book}
Neal Koblitz.
\newblock {\em {$p$}-adic numbers, {$p$}-adic analysis, and zeta-functions},
  volume~58 of {\em Graduate Texts in Mathematics}.
\newblock Springer-Verlag, New York, second edition, 1984.

\bibitem{Lazarsfeld:2004}
Robert Lazarsfeld.
\newblock {\em Positivity in algebraic geometry. {I}}, volume~48 of {\em
  Ergebnisse der Mathematik und ihrer Grenzgebiete. 3. Folge. A Series of
  Modern Surveys in Mathematics [Results in Mathematics and Related Areas. 3rd
  Series. A Series of Modern Surveys in Mathematics]}.
\newblock Springer-Verlag, Berlin, 2004.
\newblock Classical setting: line bundles and linear series.

\bibitem{Lech:1953}
Christer Lech.
\newblock A note on recurring series.
\newblock {\em Ark. Mat.}, 2:417--421, 1953.

\bibitem{NguyenBD:2017}
Bac-Dang Nguyen.
\newblock Degrees of iterates of rational transformations of projective
  varieties.
\newblock {\em arXiv}, arXiv:1701.07760:1--46, 2017.

\bibitem{Poonen:2014}
Bjorn Poonen.
\newblock {$p$}-adic interpolation of iterates.
\newblock {\em Bull. Lond. Math. Soc.}, 46(3):525--527, 2014.

\bibitem{TTTruong}
Tuyen Trung~Truong.
\newblock Relative dynamical degrees of correspondances over fields of
  arbitrary characteristic.
\newblock {\em J. Reine Angew. Math.}, to appear:1--44, 2018.

\bibitem{Urech}
Christian Urech.
\newblock Remarks on the degree growth of birational transformations.
\newblock {\em Math. Res. Lett.}, 25(1):291--308, 2018.

\end{thebibliography}

\end{document}